\documentclass[12pt]{amsart}
\usepackage{mathrsfs}
\usepackage{amsfonts}

\newtheorem{theorem}{Theorem}[section]
\newtheorem{lemma}[theorem]{Lemma}

\theoremstyle{definition}

\newtheorem{example}[theorem]{Example}

 \theoremstyle{remark}
\newtheorem{remark}[theorem]{Remark}

 \numberwithin{equation}{section}

\begin{document}

\title[A  Trudinger-Moser inequality on  bounded and convex  planar domains]{ A  sharp Trudinger-Moser inequality on  any  bounded and convex   planar domain}

\author{Guozhen Lu}
\address{Department of Mathematics, Wayne State University, Detroit, MI 48202, USA}

\email{gzlu@wayne.edu}

\author{Qiaohua  Yang}
\address{School of Mathematics and Statistics, Wuhan University, Wuhan, 430072, People's Republic of China}

\email{qhyang.math@gmail.com}

\thanks{The first author was partly supported by a US NSF grant and a Simons Fellowship from the Simons Foundation. The second author was partly   supported by   the National Natural
Science Foundation of China (No.11201346).}






\begin{abstract}

Wang and Ye conjectured in \cite{wy}:

\medskip

\emph{Let
$\Omega$ be a regular, bounded and convex domain in $\mathbb{R}^{2}$.
There exists a finite constant $C({\Omega})>0$ such that
\[
\int_{\Omega}e^{\frac{4\pi u^{2}}{H_{d}(u)}}dxdy\le C(\Omega),\;\;\forall u\in C^{\infty}_{0}(\Omega),
\]
where $H_{d}=\int_{\Omega}|\nabla u|^{2}dxdy-\frac{1}{4}\int_{\Omega}\frac{u^{2}}{d(z,\partial\Omega)^{2}}dxdy$ and $d(z,\partial\Omega)=\min\limits_{z_{1}\in\partial\Omega}|z-z_{1}|$.}\\

The main purpose of this paper is to confirm that this conjecture indeed holds for any bounded and convex domain in $\mathbb{R}^{2}$  via  the Riemann mapping theorem (the smoothness of the boundary of the domain is thus irrelevant).

 We also give a rearrangement-free argument for the following  Trudinger-Moser inequality  on the hyperbolic space $\mathbb{B}=\{z=x+iy:|z|=\sqrt{x^{2}+y^{2}}<1\}$:
\[
\sup_{\|u\|_{\mathcal{H}}\leq 1} \int_{\mathbb{B}}(e^{4\pi u^{2}}-1-4\pi u^{2})dV=\sup_{\|u\|_{\mathcal{H}}\leq 1}\int_{\mathbb{B}}\frac{(e^{4\pi u^{2}}-1-4\pi u^{2})}{(1-|z|^{2})^{2}}dxdy< \infty,
\]
by using the method employed earlier by Lam and the first author \cite{ll, ll2}, where $\mathcal{H}$ denotes the closure of $C^{\infty}_{0}(\mathbb{B})$ with respect to the
norm $$\|u\|_{\mathcal{H}}=\int_{\mathbb{B}}|\nabla u|^{2}dxdy-\int_{\mathbb{B}}\frac{u^{2}}{(1-|z|^{2})^{2}}dxdy.$$
Using this strengthened Trudinger-Moser inequality, we also give a simpler proof of the  Hardy-Moser-Trudinger inequality obtained by  Wang and  Ye \cite{wy}.

\end{abstract}

\maketitle


\section{Introduction}

 As a borderline case of the Sobolev embedding   $W^{1,p}_0(\Omega)\subset L^q(\Omega)$ where $p<N$ when $\Omega\subset \mathbb{R}^N$ $  (N\ge2)$ is a bounded domain with $1\le q\le\frac{Np}{N-p}$,  Trudinger \cite{t} proved   that $W^{1, N}_0(\Omega)\subset L_{\varphi_N}(\Omega)$,  where $L_{\varphi_N}(\Omega)$ is the Orlicz space associated with the Young function $\varphi_N(t)=\exp(\beta|t|^{N/N-1})-1$ for some $\beta>0$ (see also Yudovich \cite{Yu}, Pohozaev \cite{Po}).  J. Moser proved the following sharp result in his 1971 paper \cite{Mo}:\\
\\
\textbf{Theorem A.} \emph{Let $\Omega$ be a domain with finite measure in Euclidean N-space $\mathbb{R}^N$, $N\ge2$. Then there exists a sharp constant $\alpha_{N}=N\left(
\frac{N\pi^{\frac{N}{2}}}{\Gamma(\frac{N}{2}+1)}\right)  ^{\frac{1}{N-1}}%
$ such that
\begin{displaymath}
\frac{1}{|\Omega|}\int_{\Omega}\exp(\beta|u|^{\frac{N}{N-1}})dx\le c_0
\end{displaymath}
for any $\beta\le\alpha_N$, any $u\in W^{1,N}_0(\Omega)$ with $\int_{\Omega}|\nabla u|^N dx\le1$. This constant $\alpha_N$ is sharp in the sense that if $\beta>\alpha_N$, then the above inequality can no longer hold with some $c_0$ independent of $u$}.\\

  There have been many generalizations related to the Trudinger-Moser inequality on hyperbolic spaces (see \cite{ms}, \cite{l1}, \cite{l2}, \cite{mst}). For instance,  Mancini and Sandeep \cite{ms} (see also \cite{at})
proved the following improved Trudinger-Moser inequalities on $\mathbb{B}=\{z=x+iy:|z|=\sqrt{x^{2}+y^{2}}<1\}$:
\[
\sup_{u\in W^{1,2}_{0}(\mathbb{B}),
\int_{\mathbb{B}}|\nabla
u|^{2}dxdy\leq1}\int_{\mathbb{B}}\frac{e^{4\pi
u^{2}}-1}{(1-|z|^{2})^{2}}dxdy<\infty.
\]

 In \cite{l1, l2}, the first author and Tang established independently different type of sharp   Trudinger-Moser  inequalities from \cite{mst} on the high dimensional hyperbolic spaces. Since our main focus of this paper is on the Trudinger-Moser inequality on two dimensional case, we will not discuss further here.

\medskip

 Wang and Ye \cite{wy} proved, among other results, an improved Trudinger-Moser inequality by
combining the Hardy inequality. Their result is the following
\begin{theorem}
There exists a finite constant $C_{1}>0$ such that
\[
\int_{\mathbb{B}}e^{\frac{4\pi u^{2}}{\|u\|_{\mathcal{H}}}}dxdy\le C_{1},\;\;\forall u\in C^{\infty}_{0}(\mathbb{B}),
\]
where $\|u\|_{\mathcal{H}}=\int_{\mathbb{B}}|\nabla u|^{2}dxdy-\int_{\mathbb{B}}\frac{u^{2}}{(1-|z|^{2})^{2}}dxdy$.
\end{theorem}
We note that the  proof of Theorem 1.1 in \cite{wy} depends
on Schwartz rearrangement argument.
In the same paper, they conjecture that such Hardy-Moser-Trudinger inequality holds for bounded and convex domains with smooth boundary:

\medskip

\textbf{Conjecture} (\cite{wy})
\emph{Let
$\Omega$ be a regular, bounded and convex domain in $\mathbb{R}^{2}$.
There exists a finite constant $C({\Omega})>0$ such that
\[
\int_{\Omega}e^{\frac{4\pi u^{2}}{H_{d}(u)}}dxdy\le C(\Omega),\;\;\forall u\in C^{\infty}_{0}(\Omega),
\]
where $H_{d}=\int_{\Omega}|\nabla u|^{2}dxdy-\frac{1}{4}\int_{\Omega}\frac{u^{2}}{d(z,\partial\Omega)^{2}}dxdy$ and $d(z,\partial\Omega)=\min\limits_{z_{1}\in\partial\Omega}|z-z_{1}|$.}\\

Using Theorem1.1,  Mancini,  Sandeep and Tintarev \cite{mst} proved, among other results, the following strengthened Trudinger-moser inequality on $\mathbb{B}$
and their proof also depends on symmetrization argument.
\begin{theorem}
There exists a constant $C_{2}$ such that for all
$u\in C^{\infty}_{0}(\mathbb{B})$ with
\[
\|u\|_{\mathcal{H}}=\int_{\mathbb{B}}|\nabla u|^{2}dxdy-\int_{\mathbb{B}}\frac{u^{2}}{(1-|z|^{2})^{2}}dxdy\leq1,
\]
there holds
\[
\int_{\mathbb{B}}\frac{(e^{4\pi u^{2}}-1-4\pi u^{2})}{(1-|z|^{2})^{2}}dxdy\leq C_{2}.
\]
\end{theorem}

 Recently, Lam and the first author \cite{ll} develop a new approach to establish sharp Trudinger-Moser inequalities in unbounded domains in the settings (e.g.  Heisenberg groups) where the classical symmetrization argument does not work.
Such an approach   avoids using the
rearrangement argument which is not available in an optimal way on the Heisenberg group and can be used
in other settings such as high order Sobolev spaces, hyperbolic spaces, non-compact and complete Riemanian manifolds,  etc (see e.g.  \cite{LiLu},  \cite{ll2}, \cite{l2},  \cite{ysk}).

\medskip

One of the aims of this paper is that, in the spirit of \cite{ll,ll2}, we give a  new approach to establish the strengthened Trudinger-Moser inequality and Hardy-Moser-Trudinger inequality on $\mathbb{B}$. Our approach is much simpler and also avoids using the
rearrangement argument.

\medskip

The second  and main aim of this paper is that, using the strengthened Trudinger-Moser inequality, we  give an affirmant answer to the conjecture
given by   Wang and  Ye via Riemann mapping theorem.
The main result of our paper is the following
\begin{theorem}
Let $\Omega$ be a proper and convex domain in $\mathbb{R}^{2}$ and $u\in C^{\infty}_{0}(\Omega)$ be such that
\[
\int_{\Omega}|\nabla_{} u|^{2}dxdy-\frac{1}{4}\int_{\Omega}\frac{u^{2}}{d(z,\partial\Omega)^{2}}dxdy\leq1.
\]
Then there exists a constant $C_{3}$ which is independent of $u$  such that
\[
\int_{\Omega}\frac{e^{4\pi u^{2}}-1-4\pi u^{2}}{d(z,\partial\Omega)^{2}}dxdy\leq C_{3}.
\]
Furthermore, if  $\Omega$ is  bounded,
then there exists a constant $C_{4}$ which is independent of $u$ such that
\[
\int_{\Omega}e^{4\pi u^{2}}dxdy\leq C_{4}.
\]
\end{theorem}

\section{conformal map}
Let $U$ and $V$ be two open sets in $\mathbb{C}$. A bijection holomorphic function $f:U\rightarrow V$ is called a conformal map or a biholomorphism. We say that $U$ and $V$ are conformally equivalent or simply biholomorphic if such a conformal map $f$ exists.  Here we give  a number of special examples of conformal mappings.

\begin{example}
 Denote by $\mathbb{H}$ the upper half-plane which consists of those complex numbers with positive imaginary parts; that is
\[
\mathbb{H}=\{z=x+iy\in\mathbb{C}:y>0\}.
\]
It is known that $\mathbb{H}$ is conformally equivalent to the unit disc $\mathbb{B}$. In fact, $f:\mathbb{H}\rightarrow \mathbb{B}$ defined by
\[
f(z)=\frac{z-i}{z+i}.
\]
is a conformal map.
\end{example}

\begin{example}
For each $\alpha \in\mathbb{C}$ with $|\alpha|<1$, denote by
\[
\psi_{\alpha}(z)=\frac{\alpha-z}{1-\overline{\alpha}z}.
\]
Then $\psi_{\alpha}:\mathbb{B}\rightarrow\mathbb{B}$ is a conformal map. It is easy to check $\psi_{\alpha}(\alpha)=0$ and $\psi_{\alpha}(0)=\alpha$. Furthermore,   if $f$ is an automorphism of $\mathbb{B}$,
then there exists $\theta\in \mathbb{R}$ and $\alpha\in \mathbb{B}$ such that $f=e^{i\theta}\psi_{\alpha}$.

\end{example}

Now we recall the Riemann mapping theorem, which states that any non-empty open simply connected proper subset of $\mathbb{C}$ admits a bijective conformal map to the open unit disk $\mathbb{B}$. We state it as follows:
\begin{theorem}
Suppose $\Omega$ is proper and simply connected. If $z_{0}\in \Omega$, then there exists a unique conformal map $F:\Omega\rightarrow \mathbb{B}$ such that
\[F(z_{0})=0,\;\;\;\;\textrm{and}\;\;\;\;F'(z_{0})>0.\]
\end{theorem}

\begin{remark}
It has been shown in \cite{a} that
if $f:\Omega\rightarrow \mathbb{B}$ holomorphic, injective, $f(z_{0})=0$ and $f'(z_{0})>0$, then $F'(z_{0})\geq f'(z_{0})$.
\end{remark}

\section{ hyperbolic space of dimension two }
Recall that the Poincar\'e conformal
disc model of dimension two is the unit ball
\[\mathbb{B}=\{z=x+iy\in\mathbb{C}: |z|<1\}\]
equipped with the usual Poincar\'e metric
\[
ds^{2}=\frac{4(dx^{2}+dy^{2})}{(1-|z|^{2})^{2}}.
\]
The hyperbolic volume element is
\[
dV=\left(\frac{2}{1-|z|^{2}}\right)^{2}dxdy.\]The associated
Laplace-Beltrami operator is given by
\[
\Delta_{\mathbb{H}}=\frac{(1-|z|^{2})^{2}}{4}\left(\frac{\partial^{2}}{\partial
x^{2}}+\frac{\partial^{2}}{\partial
y^{2}}\right)
\]
and the corresponding hyperbolic gradient is
\[
\nabla_{\mathbb{H}}=\frac{1-|z|^{^{2}}}{2}\nabla=\frac{1-|z|^{^{2}}}{2}\left(\frac{\partial}{\partial
x},\frac{\partial}{\partial y}\right).
\]
It is easy to check that
\[
\int_{\mathbb{B}}|\nabla u|^{2}dxdy-\int_{\mathbb{B}}\frac{u^{2}}{(1-|z|^{2})^{2}}dxdy=\int_{\mathbb{B}}|\nabla_{\mathbb{H}} u|^{2}dV-\frac{1}{4}\int_{\mathbb{B}}u^{2}dV.
\]

For $z_{1},z_{2}\in\mathbb{B}^{n}$, we denote by $\rho(z_{1},z_{2})$ the associated
distance  from $z_{1}$ to $z_{2}$ in $\mathbb{B}^{n}$. It is well known that
\begin{equation}\label{3.1}
\begin{split}
\rho(z_{1},z_{2})=2\tanh^{-1}\left(\frac{|z_{1}-z_{2}|}{\sqrt{1-2\textrm{Re}(z_{1}\overline{z_{2}})+|z_{1}|^{2}|z_{2}|^{2}}}\right)^{\frac{1}{2}}.
\end{split}
\end{equation}
In particular, if $z_{2}=0$, then $\rho(z_{1},0)=\frac{1}{2}\log\frac{1+|z_{1}|}{1-|z_{1}|}$.
Furthermore, the polar coordinates associated with $\rho$ is
\[
\int_{\mathbb{B}}fdxdy=\int^{\infty}_{0}\int^{2\pi}_{0}f\sinh\rho d\rho d\theta,\;\;f\in L^{1}(\mathbb{B}).
\]

Let $e^{t\Delta_{\mathbb{H}}}$ be the heat kernel on $\mathbb{B}$. It is known that $e^{t\Delta_{\mathbb{H}}}$ depends only on $t$ and $\rho$. The explicit formula is (see e.g. \cite{d})
\[
h(t,\rho):=e^{t\Delta_{\mathbb{H}}}=\frac{\sqrt{2}}{8\pi^{\frac{3}{2}}}t^{-\frac{3}{2}}e^{-\frac{t}{4}}\int^{+\infty}_{\rho}\frac{re^{-\frac{r^{2}}{4t}}}{\sqrt{\cosh r-\cosh\rho}}dr.
\]
Via the heat kernel, the fractional power
\begin{equation}\label{3.2}
\begin{split}
(-\Delta_{\mathbb{H}}-1/4)^{-\frac{1}{2}}=&\frac{1}{\Gamma(1/2)}\int^{+\infty}_{0}t^{-\frac{1}{2}}e^{t(\Delta_{\mathbb{H}}+1/4)}dt\\
=&\frac{\sqrt{2}}{8\pi^{2}}\int^{+\infty}_{\rho}\frac{r}{\sqrt{\cosh r-\cosh\rho}}dr\int^{+\infty}_{0}t^{-2}e^{-\frac{r^{2}}{4t}}dt\\
=&\frac{\sqrt{2}}{2\pi^{2}}\int^{+\infty}_{\rho}\frac{1}{r\sqrt{\cosh r-\cosh\rho}}dr.
\end{split}
\end{equation}

\begin{remark}
The operator $-\Delta_{\mathbb{H}}-1/4$ has been studied earlier  by Beckner \cite{be}. In fact, it has been shown by Beckner that for $F\in C^{\infty}_{0}(M)$,
\begin{equation}\label{1.6}
\left[\|F\|_{L^{6}(M)}\right]^{2}\leq 4\pi^{-2/3}\left[\int_{M}|DF|^{2}d\nu-\frac{3}{16}\int_{M}F^{2}d\nu\right];
\end{equation}
\begin{equation}\label{1.7}
\left[\|F\|_{L^{6}(M)}\right]^{2}\leq \frac{4}{3}\pi^{-2/3}\left[\int_{M}|DF|^{2}d\nu-\frac{1}{4}\int_{M}F^{2}d\nu\right];
\end{equation}
\begin{equation}\label{1.8}
\left[\|F\|_{L^{4}(M)}\right]^{2}\leq 2\pi^{-1/2}\left[\int_{M}|DF|^{2}d\nu-\frac{1}{4}\int_{M}F^{2}d\nu\right],
\end{equation}
where $M$ is the half space model of two-dimensional hyperbolic space, $DF=\frac{1}{y}\nabla F$ and $d\nu=\frac{1}{y^{2}}dxdy$.
It seems that inequalities (\ref{1.6}) and (\ref{1.7}) would  be contradictory since both estimates are sharp as limiting forms. However, as pointed out by Beckner, the right-hand side of inequality (\ref{1.7})  is to be evaluated as limiting forms for functions that
may not be in $L^{2}(M)$.
\end{remark}

\begin{lemma} Set $\phi(\rho)=(-\Delta_{\mathbb{H}}-1/4)^{-\frac{1}{2}}$. Then for $\rho>0$,
\begin{equation}\label{3.3}
\begin{split}
\phi(\rho)\leq\frac{1}{4\pi\sinh\frac{\rho}{2}};
\end{split}
\end{equation}
\begin{equation}\label{3.4}
\begin{split}
\phi(\rho)\leq\frac{1}{2\pi\rho\sinh\frac{\rho}{2}}.
\end{split}
\end{equation}
\end{lemma}
\begin{proof} Since for $\rho>0$, $\sinh\frac{\rho}{2}\leq\frac{\rho}{2}\cosh\frac{\rho}{2}$, we have, for $\rho>0$,
\begin{equation}\label{3.5}
\begin{split}
&\int^{+\infty}_{\rho}\frac{1}{r\sqrt{\cosh r-\cosh\rho}}dr\leq\int^{+\infty}_{\rho}\frac{\cosh\frac{r}{2}}{2\sinh\frac{r}{2}\sqrt{\cosh r-\cosh\rho}}dr\\
=&\int^{+\infty}_{\rho}\frac{\cosh\frac{r}{2}}{2\sinh\frac{r}{2}\sqrt{2\sinh^{2}\frac{r}{2}-2\sinh^{2}\frac{\rho}{2}}}dr\\
=&\frac{1}{\sqrt{2}\sinh\frac{\rho}{2}}\arctan\left(\sinh^{2}\frac{r}{2}-\sinh^{2}\frac{\rho}{2}\right)\left|^{+\infty}_{\rho}\right.\\
=&\frac{\pi}{2\sqrt{2}\sinh\frac{\rho}{2}}.
\end{split}
\end{equation}
Combing (\ref{3.2}) and (\ref{3.5}) yields (\ref{3.3}).

\medskip

Similarly, using the fact $\sinh\frac{\rho}{2}\leq\cosh\frac{\rho}{2}$, we have
\begin{equation}\label{3.6}
\begin{split}
&\int^{+\infty}_{\rho}\frac{1}{r\sqrt{\cosh r-\cosh\rho}}dr\leq\frac{1}{\rho}\int^{+\infty}_{\rho}\frac{1}{\sqrt{\cosh r-\cosh\rho}}dr\\
\leq&\frac{1}{\rho}\int^{+\infty}_{\rho}\frac{\cosh\frac{r}{2}}{\sinh\frac{r}{2}\sqrt{\cosh r-\cosh\rho}}dr\\
=&\frac{1}{\rho}\int^{+\infty}_{\rho}\frac{\cosh\frac{r}{2}}{\sinh\frac{r}{2}\sqrt{2\sinh^{2}\frac{r}{2}-2\sinh^{2}\frac{\rho}{2}}}dr\\
=&\frac{\sqrt{2}}{\rho\sinh\frac{\rho}{2}}\arctan\left(\sinh^{2}\frac{r}{2}-\sinh^{2}\frac{\rho}{2}\right)\left|^{+\infty}_{\rho}\right.\\
=&\frac{\sqrt{2}\pi}{2\rho\sinh\frac{\rho}{2}}.
\end{split}
\end{equation}
Combing (\ref{3.2}) and (\ref{3.6}) yields (\ref{3.4}).
\end{proof}

We now recall the rearrangement of a real functions on $\mathbb{B}$.  Suppose $f$ is
a real  function on $\mathbb{B}$. The non-increasing rearrangement of $f$
is defined by
\begin{equation}\label{3.7}
f^{\ast}(t)=\inf\{s>0: \lambda_{f}(s)\leq t\},
\end{equation}
where $$\lambda_{f}(s)=|\{z\in \mathbb{B}: |f(z)|>s\}|=\int_{\{z\in \mathbb{B}:: |f(z)|>s\}}\left(\frac{2}{1-|z|^{2}}\right)^{2}dxdy.$$
 Here we use the
notation $|\Sigma|$ for the measure of a measurable set
$\Sigma\subset \mathbb{B}$.

\begin{lemma}
Set $\phi(\rho)=(-\Delta_{\mathbb{H}}-1/4)^{-\frac{1}{2}}$.  Then, for $t>0$,
\begin{equation}\label{3.8}
\phi^{\ast}(t)\leq \frac{1}{\sqrt{4\pi t}}
\end{equation}
and, for each $a>0$,
\begin{equation}\label{3.9}
\int^{\infty}_{a}|\phi^{\ast}(t)|^{2}dt<\infty.
\end{equation}

\end{lemma}
\begin{proof}
Define, for any $s>0$,
\begin{equation}\label{3.10}
\lambda_{\phi}(s)=\int_{\{\phi(\rho)>s\}}dV=2\pi\int^{\rho_{s}}_{0}\sinh\rho d\rho,
\end{equation}
where $\rho_{s}$ is the solution of equation
\begin{equation}\label{3.11}
\phi(\rho)=s.
\end{equation}
Therefore, since $\phi^{\ast}(t)=\inf\{s>0: \lambda_{\phi}(s)\leq t\}$, we have
\begin{equation}\label{3.12}
t=\lambda_{\phi}(\phi^{\ast}(t))=2\pi\int^{\rho_{\phi^{\ast}(t)}}_{0}\sinh\rho d\rho=2\pi(\cosh\rho_{\phi^{\ast}(t)}-1),
\end{equation}
where $\rho_{g^{\ast}(t)}$ satisfies
\begin{equation}\label{3.13}
\phi( \rho_{\phi^{\ast}(t)})=\phi^{\ast}(t).
\end{equation}
By (\ref{3.3})
\begin{equation}\label{3.14}
\begin{split}
\phi^{\ast}(t)=&\phi( \rho_{\phi^{\ast}(t)})\leq\frac{1}{4\pi\sinh\frac{\rho_{\phi^{\ast}(t)}}{2}}.
\end{split}
\end{equation}
Combing (\ref{3.12}) and (\ref{3.14}) yields
\begin{equation}\label{3.15}
\begin{split}
t|\phi^{\ast}(t)|^{2}\leq\frac{2\pi(\cosh\rho_{\phi^{\ast}(t)}-1)}{16\pi^{2}\sinh^{2}\frac{\rho_{\phi^{\ast}(t)}}{2}}=\frac{1}{4\pi}.
\end{split}
\end{equation}
This proves inequality (\ref{3.8}).

\medskip

Now we prove (\ref{3.9}). Using the substitution $t=2\pi(\cosh\rho_{\phi^{\ast}(t)}-1)$, we have, by (\ref{3.4}) and (\ref{3.13}),
\begin{equation}\label{3.16}
\begin{split}
\int^{\infty}_{a}|\phi^{\ast}(t)|^{2}dt=&\int^{\infty}_{b}|\phi( \rho_{\phi^{\ast}(t)})|^{2}2\pi\sinh \rho_{\phi^{\ast}(t)}d \rho_{\phi^{\ast}(t)}\\
\leq& \int^{\infty}_{b}\frac{\sinh \rho_{\phi^{\ast}(t)}}{\rho^{2}_{\phi^{\ast}(t)}\sinh^{2} \frac{\rho_{\phi^{\ast}(t)}}{2}}d \rho_{\phi^{\ast}(t)}=\int^{\infty}_{b}\frac{\sinh s}{s^{2}\sinh^{2} \frac{s}{2}}ds\\
=&2\int^{\infty}_{b}\frac{\cosh \frac{s}{2}}{s^{2}\sinh \frac{s}{2}}ds,
\end{split}
\end{equation}
where $b>0$ satisfies $a=2\pi(\cosh b-1)$.
Since $\lim\limits_{s\rightarrow+\infty}\frac{\cosh \frac{s}{2}}{\sinh \frac{s}{2}}=1$, there exists a positive constant $C_{b}$ such that
$\frac{\cosh \frac{s}{2}}{\sinh \frac{s}{2}}\leq C_{b}$ for all $s\in [b,+\infty)$.
Therefore, by (\ref{3.16}),
\[
\int^{\infty}_{a}|\phi^{\ast}(t)|^{2}dt\leq2C_{b}\int^{\infty}_{b}\frac{1}{s^{2}}ds=\frac{2C_{b}}{b}.
\]
The desired result follows.
\end{proof}

\section{proof of Theorem 1.1 and Theorem 1.2 }

Before we prove the theorems, we need the following lemma from
Adams' paper \cite{ad}.

\begin{lemma}
Let  $a(s,t)$ be a non-negative measurable function on
$(-\infty,+\infty)\times[0,+\infty)$ such that (a.e.)
\[
a(s,t)\leq1,\;\;when\;\;0<s<t,
\]
\[
\sup_{t>0}\left(\int^{0}_{-\infty}a(s,t)^{n'}ds+\int^{\infty}_{t}a(s,t)^{n'}ds\right)^{1/n'}=b<\infty,
\]
where $n'=\frac{n}{n-1}$. Then there is a constant
$c_{0}=c_{0}(n,b)$ such that if for $\phi\geq0$ with
$\int^{\infty}_{-\infty}\phi(s)^{n}ds\leq1$, then
\[
\int^{\infty}_{0}e^{-F(t)}dt\leq c_{0},
\]
where
\[
F(t)=t-\left(\int^{\infty}_{-\infty}a(s,t)\phi(s)ds\right)^{n'}.
\]
\end{lemma}

Now we prove Theorem 1.2. The main idea is to adapt the level set developed by Lam and the first author to derive a global Trudinger-Moser inequality
from a local one (see \cite{ll,ll2}). We firstly prove Theorem 1.2.\\

\textbf{Proof of Theorem 1.2}  Let $u\in C^{\infty}_{0}(\mathbb{B})$ be such that
$$\int_{\mathbb{B}}|\nabla u|^{2}dxdy-\int_{\mathbb{B}}\frac{u^{2}}{(1-|z|^{2})^{2}}dxdy=\int_{\mathbb{B}}|\nabla_{\mathbb{H}} u|^{2}dV-\frac{1}{4}\int_{\mathbb{B}}u^{2}dV\leq1.$$
 Set $\Omega(u)=\{z\in\mathbb{B}:|u(z)|\geq1\}$. By inequality (\ref{1.8}),
\begin{equation}\label{4.1}
\begin{split}
|\Omega(u)|=&\int_{\Omega(u)}dV\leq\int_{\mathbb{B}}|u(z)|^{4}dV\\
\leq& 4\pi^{-1}\left(\int_{\mathbb{B}}|\nabla_{\mathbb{H}} u|^{2}dV-\frac{1}{4}\int_{\mathbb{B}}u^{2}dV\right)^{2}\\
\leq& 4\pi^{-1}.
\end{split}
\end{equation}

we write
\begin{equation}\label{4.2}
\begin{split}
&\int_{\mathbb{B}}(e^{4\pi u^{2}}-1-4\pi u^{2})dV\\
=&\int_{\Omega(u)}(e^{4\pi u^{2}}-1-4\pi u^{2})dV+
\int_{\mathbb{B}\setminus\Omega(u)}(e^{4\pi u^{2}}-1-4\pi u^{2})dV\\
\leq&\int_{\Omega(u)}e^{4\pi u^{2}}dV+
\int_{\mathbb{B}\setminus\Omega(u)}(e^{4\pi u^{2}}-1-4\pi u^{2})dV.
\end{split}
\end{equation}
Notice that on the domain $\mathbb{B}\setminus\Omega(u)$, we have $|u(z)|<1$. Thus, by (\ref{4.1}) and (\ref{1.8}),
\begin{equation}\label{4.3}
\begin{split}
\int_{\mathbb{B}\setminus\Omega(u)}(e^{4\pi u^{2}}-1-4\pi u^{2})dV=&\int_{\mathbb{B}\setminus\Omega(u)}\sum^{\infty}_{n=2}\frac{(4\pi u^{2})^{n}}{n!}dV\\
\leq&\int_{\mathbb{B}\setminus\Omega(u)}\sum^{\infty}_{n=2}\frac{(4\pi )^{n}u^{4}}{n!}dV\\
\leq&\sum^{\infty}_{n=2}\frac{(4\pi )^{n}}{n!}\int_{\mathbb{B}}|u(z)|^{4}dV\\
\leq&e^{4\pi} 4\pi^{-1}\left(\int_{\mathbb{B}}|\nabla_{\mathbb{H}} u|^{2}dV-\frac{1}{4}\int_{\mathbb{B}}u^{2}dV\right)^{2}\\
\leq&4\pi^{-1}e^{4\pi}.
\end{split}
\end{equation}

To finish the proof, it is enough to show $\int_{\Omega(u)}e^{4\pi u^{2}}dV$ is bounded by some universal constant.
We rewrite
\begin{equation}\label{b4.1}
\begin{split}
\int_{\mathbb{B}}|\nabla_{\mathbb{H}} u|^{2}dV-\frac{1}{4}\int_{\mathbb{B}}u^{2}dV=\int_{\mathbb{B}}\left|\left(-\Delta_{\mathbb{H}}-\frac{1}{4}\right)^{\frac{1}{2}} u\right|^{2}dV\leq 1.
\end{split}
\end{equation}
Therefore, we can write $u$ as a potential via (\ref{3.2})
\begin{equation}\label{4.4}
\begin{split}
u(z)=&\int_{\mathbb{B}}v(\omega)\left(\frac{\sqrt{2}}{2\pi^{2}}\int^{+\infty}_{\rho(\omega,z)}\frac{1}{r\sqrt{\cosh r-\cosh\rho(\omega,z)}}dr\right)dV_{\omega}
\\
=&\int_{\mathbb{B}}v(\omega)\phi(\rho(\omega,z))dV_{\omega},
\end{split}
\end{equation}
where
\[
\phi(\rho(\omega,z))=\frac{\sqrt{2}}{2\pi^{2}}\int^{+\infty}_{\rho(\omega,z)}\frac{1}{r\sqrt{\cosh r-\cosh\rho(\omega,z)}}dr.
\]
Furthermore, by (\ref{b4.1}),  $v\in L^{2}(\mathbb{B})$ with $\int_{\mathbb{B}}v^{2}dV\leq1$.

\medskip

By O'Neil's lemma and (\ref{4.4}), we have, for all $t>0$,
\begin{equation}\label{4.5}
\begin{split}
u^{\ast}(t)\leq\frac{1}{t}\int^{t}_{0}v^{\ast}(s)ds\int^{t}_{0}\phi^{\ast}(s)ds+\int^{\infty}_{t}v^{\ast}(s)\phi^{\ast}(s)ds.
\end{split}
\end{equation}
Then, since $|\Omega(u)|\leq4\pi^{-1}<4$, we have, by (\ref{4.5})
\begin{equation}\label{4.6}
\begin{split}
&\int_{\Omega(u)}e^{4\pi u^{2}}dV=\int^{|\Omega(u)|}_{0}\exp(4\pi|u^{\ast}(t)|^{2})dt\\
\leq&\int^{4}_{0}\exp(4\pi|u^{\ast}(t)|^{2})dt\\
\leq&\int^{4}_{0}\exp\left(4\pi\left|\frac{1}{t}\int^{t}_{0}v^{\ast}(s)ds\int^{t}_{0}\phi^{\ast}(s)ds+
\int^{\infty}_{t}v^{\ast}(s)\phi^{\ast}(s)ds\right|^{2}\right)dt
\\
=&4\int^{\infty}_{0}\exp\left(-t+\frac{\pi}{4}\left|e^{t}\int^{4e^{-t}}_{0}v^{\ast}(s)ds\int^{4e^{-t}}_{0}\phi^{\ast}(s)ds+
4\int^{\infty}_{4e^{-t}}v^{\ast}(s)\phi^{\ast}(s)ds\right|^{2}\right)dt.
\end{split}
\end{equation}
To get the last equation, we use the substitution $t:=4e^{-t}$.
Next,  we change the variables
\begin{equation*}
\begin{split}
\psi(t)=&2e^{-t/2}v^{\ast}(4e^{-t});\\
\varphi(t)=&4\sqrt{\pi}e^{-t/2}\phi^{\ast}(4e^{-t}).
\end{split}
\end{equation*}
it is easy to check
\begin{equation}\label{4.7}
\begin{split}
e^{t}\int^{\infty}_{t}e^{-s/2}\psi(s)ds\int^{\infty}_{t}e^{-s/2}\varphi(s)ds=&\frac{\sqrt{\pi}}{2}e^{t}
\int^{4e^{-t}}_{0}v^{\ast}(s)ds\int^{4e^{-t}}_{0}\phi^{\ast}(s)ds;\\
\int^{t}_{-\infty}\psi(s)\varphi(s)ds=&2\sqrt{\pi}\int^{\infty}_{4e^{-t}}v^{\ast}(s)\phi^{\ast}(s)ds.
\end{split}
\end{equation}
Combing (\ref{4.6}) and (\ref{4.7}) yields
\begin{equation}\label{4.8}
\begin{split}
&\int_{\Omega(u)}e^{4\pi u^{2}}dV\leq\int^{4}_{0}\exp(4\pi|u^{\ast}(t)|^{2})dt\\
=&4\int^{\infty}_{0}\exp\left(-t+\frac{\pi}{4}\left|e^{t}\int^{4e^{-t}}_{0}v^{\ast}(s)ds\int^{4e^{-t}}_{0}\phi^{\ast}(s)ds+
4\int^{\infty}_{4e^{-t}}v^{\ast}(s)\phi^{\ast}(s)ds\right|^{2}\right)dt\\
=&4\int^{\infty}_{0}e^{-F(t)}dt,
\end{split}
\end{equation}
where
\[
F(t)=t-\left(e^{t}\int^{\infty}_{t}e^{-s/2}\psi(s)ds\int^{\infty}_{t}e^{-s/2}\varphi(s)ds+\int^{t}_{-\infty}\psi(s)\varphi(s)ds\right)^{2}.
\]
Using Lemma 3.2, we have
\begin{equation*}
\begin{split}
\sup_{s>0}\varphi(s)=&4\sup_{s>0}\{\sqrt{\pi}e^{-s/2}\phi^{\ast}(4e^{-s})\}\leq4\sup_{s>0}\left\{\sqrt{\pi}e^{-s/2}\frac{1}{\sqrt{4\pi4e^{-s}}}\right\}=1;\\
\int^{0}_{-\infty}|\psi(s)|^{2}ds=&16\pi\int^{0}_{-\infty}e^{-t}|\phi^{\ast}(4e^{-t})|^{2}dt=4\pi\int^{\infty}_{4}|\phi^{\ast}(t)|^{2}dt<\infty.
\end{split}
\end{equation*}
Furthermore,
\[
\int^{+\infty}_{-\infty}|\psi(s)|^{2}ds=\int^{\infty}_{0}|v^{\ast}(s)|^{2}ds=\int_{\mathbb{B}}|v(z)|^{2}dV\leq1.
\]
Thus, if we set
\[
a(s,t)=\left\{
         \begin{array}{ll}
           \varphi(s), & \hbox{$s<t$;} \\
           e^{t}(\int^{\infty}_{t}e^{-s/2}\varphi(s)ds)e^{-s/2}, & \hbox{$s>t$,}
         \end{array}
       \right.
\]
then by Lemma 4.1,
$\int_{\Omega(u)}e^{4\pi u^{2}}dV$ is bounded by some  constant which is independent of $u$ and $\Omega(u)$.
The proof of Theorem 1.2 is thereby completed.\\

\medskip

Now we can prove Theorem 1.1 via Theorem 1.2.

\textbf{Proof of Theorem 1.1} Let $u\in C^{\infty}_{0}(\mathbb{B})$ be such that
$$\int_{\mathbb{B}}|\nabla u|^{2}dxdy-\int_{\mathbb{B}}\frac{u^{2}}{(1-|z|^{2})^{2}}dxdy=\int_{\mathbb{B}}|\nabla_{\mathbb{H}} u|^{2}dV-\frac{1}{4}\int_{\mathbb{B}}u^{2}dV\leq1.$$ By Theorem 1.2, there exist a positive constant $C_{2}$ which is independent of $u$ such that
\[
\int_{\mathbb{B}}\frac{(e^{4\pi u^{2}}-1-4\pi u^{2})}{(1-|z|^{2})^{2}}dxdy\leq C_{2}.
\]
Therefore,
\begin{equation*}
\begin{split}
\int_{\mathbb{B}}e^{4\pi u^{2}}dxdy=&\int_{\mathbb{B}}(e^{4\pi u^{2}}-1-4\pi u^{2})dxdy++\int_{\mathbb{B}}dxdy+4\pi\int_{\mathbb{B}}u^{2}dxdy\\
\leq&\int_{\mathbb{B}}\frac{(e^{4\pi u^{2}}-1-4\pi u^{2})}{(1-|z|^{2})^{2}}dxdy+\int_{\mathbb{B}}dxdy+4\pi\int_{\mathbb{B}}u^{2}dxdy\\
\leq& C_{2}+|\mathbb{B}|+4\pi\cdot C^{-1}.
\end{split}
\end{equation*}
To get the last inequality, we use the improved Hardy inequality (see e.g. \cite{bm,wy})
\[
\int_{\mathbb{B}}|\nabla u|^{2}dxdy-\int_{\mathbb{B}}\frac{u^{2}}{(1-|z|^{2})^{2}}dxdy\geq C\int_{\mathbb{B}}u^{2}dxdy.
\]
The desired result follows.

\section{proof of  Theorem 1.3}

Since a convex domain in $\mathbb{R}^{2}$ is also simply connected, we have, by Riemann mapping theorem, there exists a conformal map $F:\Omega\rightarrow \mathbb{B}$. Therefore, by Theorem 1.2, Then there exists a constant $C_{3}$ such that for all
$u\in C^{\infty}_{0}(\Omega)$ with
\begin{equation}\label{b5.1}
\int_{\Omega}|\nabla_{} u|^{2}dxdy-\int_{\Omega}u^{2}\frac{|F'(z)|^{2}}{(1-|F(z)|^{2})^{2}}dxdy\leq1,
 \end{equation}
there holds
\begin{equation}\label{b5.2}
\int_{\Omega}(e^{4\pi u^{2}}-1-4\pi u^{2})\frac{|F'(z)|^{2}}{(1-|F(z)|^{2})^{2}}dxdy\leq C_{3}.
 \end{equation}

Next we  shall show that for each
$z_{0}\in\Omega$,
\begin{equation}\label{5.1}
\frac{|F'(z_{0})|^{2}}{(1-|F(z_{0})|^{2})^{2}}\geq \frac{1}{4d(z_{0},\partial\Omega)^{2}}.
  \end{equation}
Since $\Omega$ is  proper, there exists $z_{1}\in\partial\Omega$ such that $d(z_{0},\partial\Omega)=|z_{0}-z_{1}|$. Furthermore, since $\Omega$ is convex, $\Omega$  lies in the half-plane (see e.g.\cite{ow})
\[
H_{z_{0}}:=\left\{z\in\mathbb{C}: \textrm{Re}\frac{z-z_{1}}{z_{0}-z_{1}}>0\right\}.
\]
Now we construct a holomorphic and injective from $H_{z_{0}}$ into $\mathbb{B}$. Via Example 2.1, it is easy to check
\[
f_{z_{0}}(z):=\frac{\frac{z-z_{1}}{z_{0}-z_{1}}-1}{\frac{z-z_{1}}{z_{0}-z_{1}}+1}\cdot\frac{z_{0}-z_{1}}{|z_{0}-z_{1}|}=\frac{z-z_{0}}{z+z_{0}-2z_{1}}\cdot\frac{z_{0}-z_{1}}{|z_{0}-z_{1}|}
\]
is  such a function. Furthermore,
\begin{equation}\label{5.2}
f_{z_{0}}(z_{0})=0,\;\;\;\;f'_{z_{0}}(z_{0})=\frac{1}{2|z_{0}-z_{1}|}=\frac{1}{2d(z_{0},\partial\Omega)}.
\end{equation}

Set
$$G(z)=-\frac{|F'(z_{0})|}{F'(z_{0})}\psi_{F(z_{0})}(F(z))=-\frac{|F'(z_{0})|}{F'(z_{0})}\frac{F(z_{0})-F(z)}{1-\overline{F(z_{0})}F(z)},$$
where $\psi_{F(z_{0})}$ is defined in Example 2.2. Since $F:\Omega\rightarrow \mathbb{B}$ be a conformal  map,
so does $G$. Furthermore,
\[
G(z_{0})=0,\;\;\;\;G'(z_{0})=\frac{|F'(z_{0})|}{1-|F(z_{0})|^{2}}>0.
\]
Therefore, by Remark 2.4,
$$\frac{|F'(z_{0})|}{1-|F(z_{0})|^{2}}=G'(z_{0})\geq f'_{z_{0}}(z_{0})=\frac{1}{2d(z_{0},\partial\Omega)}.$$

Let $u\in C^{\infty}_{0}(\Omega)$ be such that
\[
\int_{\Omega}|\nabla_{} u|^{2}dxdy-\frac{1}{4}\int_{\Omega}\frac{u^{2}}{d(z,\partial\Omega)^{2}}dxdy\leq1.
\]
Then
\[
\int_{\Omega}|\nabla_{} u|^{2}dxdy-\int_{\Omega}u^{2}\frac{|F'(z)|^{2}}{(1-|F(z)|^{2})^{2}}dxdy\leq\int_{\Omega}|\nabla_{} u|^{2}dxdy-\frac{1}{4}\int_{\Omega}\frac{u^{2}}{d(z,\partial\Omega)^{2}}dxdy\leq1.
\]
By (\ref{b5.1}) and (\ref{b5.2}), we have
\[
\int_{\Omega}\frac{e^{4\pi u^{2}}-1-4\pi u^{2}}{d(z,\partial\Omega)^{2}}dxdy\leq4\int_{\Omega}(e^{4\pi u^{2}}-1-4\pi u^{2})\frac{|F'(z)|^{2}}{(1-|F(z)|^{2})^{2}}dxdy\leq4 C_{3}.
\]
Furthermore,
if $\Omega$ is bounded, there exists a positive constant $M$ such that for all $z\in\Omega$, $d(z,\partial\Omega)\leq M$. Therefore,
for each
$u\in C^{\infty}_{0}(\Omega)$ with
\[
\int_{\Omega}|\nabla_{} u|^{2}dxdy-\frac{1}{4}\int_{\Omega}\frac{u^{2}}{d(z,\partial\Omega)^{2}}dxdy\leq1,
\]
we have
\begin{equation*}
\begin{split}
\int_{\Omega}e^{4\pi u^{2}}dxdy=&\int_{\Omega}(e^{4\pi u^{2}}-1-4\pi u^{2})dxdy++\int_{\Omega}dxdy+4\pi\int_{\Omega}u^{2}dxdy\\
\leq&M^{2}\int_{\Omega}\frac{(e^{4\pi u^{2}}-1-4\pi u^{2})}{d(z,\partial\Omega)^{2}}dxdy+\int_{\Omega}dxdy+4\pi\int_{\Omega}u^{2}dxdy\\
\leq& M^{2}C_{4}+|\Omega|+4\pi\cdot C^{-1}_{6}.
\end{split}
\end{equation*}
To get the last inequality, we use the improved Hardy inequality (see e.g. \cite{bm,ho})
\[
\int_{\Omega}|\nabla u|^{2}dxdy-\frac{1}{4}\int_{\Omega}\frac{u^{2}}{d(z,\partial\Omega)^{2}}dxdy\geq C_{6}\int_{\Omega}u^{2}dxdy.
\]
The desired result follows.

\end{document}